\documentclass[final,a4paper,12pt]{amsart}
\usepackage{version}
\usepackage[notcite,notref]{showkeys}
\usepackage{amssymb}
\usepackage{amsfonts}
\usepackage{graphicx}
\usepackage{tikz}
\usetikzlibrary{arrows}

\usepackage{color}
\usepackage{epstopdf}
\usepackage[english]{babel}
\usepackage{float}
\usepackage{cite}
\usepackage{tikz,pgf}
\usetikzlibrary{patterns,spy}

\usepackage{geometry}
\newtheorem{theorem}{Theorem}[section]

\newtheorem{lemma}[theorem]{Lemma}
\newtheorem{sublemma}[theorem]{Sublemma}
\newtheorem{corollary}[theorem]{Corollary}

\theoremstyle{definition}
\newtheorem{definition}[theorem]{Definition} 
\newtheorem{example}[theorem]{Example}

\theoremstyle{remark}
\newtheorem{remark}[theorem]{Remark}

\numberwithin{equation}{section}

\renewcommand\bigskip{\medskip}

\def\to{\rightarrow}

\def\N{\mathbb N}

\def\Q{\mathbb Q}
\def\R{\mathbb R}

\def\vep{\varepsilon}

\DeclareMathOperator{\dist}{dist}

\usepackage[utf8]{inputenc}

\usepackage[T1]{fontenc}         

\usepackage{amsmath}            
\usepackage{amssymb}            
\usepackage{amsfonts}           
\usepackage{mathrsfs}
                                 
\usepackage{latexsym}           
\usepackage{amsthm}             
\usepackage{url}

\usepackage{hyperref}           
\usepackage{tikz}
\usetikzlibrary{positioning}
\usetikzlibrary{arrows}
\usepackage{subcaption}
\usepackage{mathtools}         
\usepackage{latexsym}            
\usepackage{bbm}
\usepackage{tikz-cd}

\newcommand{\field}[1]{\mathbb{#1}}            		   
\newcommand{\mc}[1]{\mathcal{#1}}
\newcommand{\Pb}{\field{P}}

\DeclareMathOperator{\spt}{\mathrm{spt}}
\DeclareMathOperator{\capa}{\mathrm{Cap}}
\DeclareMathOperator{\dimh}{\mathrm{dim}_{\mathrm{H}}}
\DeclareMathOperator{\dimhu}{\overline{\mathrm{dim}}_{\mathrm{H}}}
\DeclareMathOperator{\dimhl}{\underline{\mathrm{dim}}_{\mathrm{H}}}
\DeclareMathOperator{\dimloc}{\mathrm{dim}_{\mathrm{loc}}}
\DeclareMathOperator{\dimlocu}{\overline{\mathrm{dim}}_{\mathrm{loc}}}
\DeclareMathOperator{\dimlocl}{\underline{\mathrm{dim}}_{\mathrm{loc}}}
\DeclareMathOperator*{\esssup}{ess\,sup}
\DeclareMathOperator*{\essinf}{ess\,inf}

\usepackage{graphicx}
\graphicspath{ {./images/} }

\usepackage{blindtext}

\begin{document}

\title[Ekstr\"om-Persson conjecture]{Ekstr\"om-Persson conjecture regarding
  random covering sets}

\author[E. J\"arvenp\"a\"a]{Esa J\"arvenp\"a\"a}
\address{Department of Mathematical Sciences, P.O. Box 3000, 90014 
University of Oulu, Finland}
\email{esa.Jarvenpaa@oulu.fi}

\author[M. J\"arvenp\"a\"a]{Maarit J\"arvenp\"a\"a}
\address{Department of Mathematical Sciences, P.O. Box 3000, 90014 
University of Oulu, Finland}
\email{maarit.Jarvenpaa@oulu.fi}

\author[M. Myllyoja]{Markus Myllyoja}
\address{Department of Mathematical Sciences, P.O. Box 3000, 90014 
University of Oulu, Finland}
\email{markus.myllyoja@oulu.fi}

\author[\"O. Stenflo]{\"Orjan Stenflo}
\address{Department of Mathematics, Uppsala University, P.O. Box 480, 75106
Uppsala, Sweden}
\email{stenflo@math.uu.se}

\thanks{The third author is supported by Emil Aaltonen Foundation.
 We thank Sylvester Eriksson-Bique and Ruxi Shi for interesting discussions.}


\begin{abstract}
We consider the Hausdorff dimension of random covering sets generated by balls
and general measures in Euclidean spaces. We prove, for a certain parameter
range, a conjecture by Ekstr\"om and Persson concerning the exact value
of the dimension in the special case of radii $(n^{-\alpha})_{n=1}^\infty$. For
generating balls with an arbitrary sequence of radii, we
find sharp bounds for the dimension and show that the natural extension of the
Ekstr\"om-Persson conjecture is not true in this case. Finally, we construct
examples demonstrating that there does not exist a dimension formula involving
only the lower and upper local dimensions of the measure and a critical
parameter determined by the sequence of radii.
\end{abstract}

\maketitle

\section{Introduction}\label{intro}

The limsup set $E(A_n)$ of a sequence $(A_n)_{n=1}^\infty$ of subsets
of some space $X$ consists of those points of $X$ which belongs to infinitely
many of the sets $A_n$, that is,
\[
E(A_n):=\limsup_{n\to\infty}A_n=\bigcap_{n=1}^\infty\bigcup_{k=n}^\infty A_k.
\]
A random covering set is a limsup set where the selection of the sets $A_n$
involves some randomness. For the exact setup of this paper, see
Definition~\ref{definition:randomcoveringset}. The study of random covering
sets  may be traced back to
a paper of Borel \cite{Bor1897} in the late 1890's, where he stated that a
given point of a circle belongs almost surely to infinitely many randomly placed
arcs provided the sum of their lengths is infinite. That article is the origin
of a theorem nowadays known as the Borel-Cantelli lemma. The dimensional
properties of related limsup sets are implicitly studied by Besicovitch
\cite{Bes1935} and Eggleston \cite{Egg1949} in the connection of
Besicovitch-Eggleston sets concerning $k$-adic expansions of real numbers. This
kind of limsup sets appear also in Diophantine approximation as demonstrated by
the classical theorems of Khintchine \cite{Khi1924} and Jarn\'{i}k
\cite{Jar1929}.

In this paper, we concentrate on the Hausdorff dimension of random covering
sets. For the closely related topic of shrinking target problem, we refer to
\cite{Dav2022a,Dav2022b,HV1995,HV1999,LLVZ2023,SW2013,KKP2020,KKP2021}.
The study of Hausdorff dimension of random covering sets was initiated by Fan
and Wu in \cite{FW2004}. They proved that the Hausdorff dimension of the set of
points on the unit circle covered by infinitely many randomly placed arcs is
almost surely equal to $\frac 1{\alpha}$ provided the lengths of the arcs is
given by the sequence $(n^{-\alpha})_{n=1}^\infty$. Their method can be extended to
a general sequence of lengths $\underline l=(l_n)_{n=1}^\infty$ giving that the
Hausdorff dimension of the random covering set equals almost surely
\begin{equation}\label{s2}
s_2(\underline l):=\inf\{s>0\mid\sum_{n=1}^\infty l_n^s<\infty\}
  =\sup\{s>0\mid\sum_{n=1}^\infty l_n^s=\infty\}
\end{equation}
provided this number is at most 1.  
This result was also proved by Durand \cite{Dur2010} using a different method.
The mass transference principle, also known as the ubiquity theorem
(see \cite{BV2006,Dav2022b,E-B2022,Jaf2000,KR2021}), can be applied to
extend the above result to randomly placed balls on the $d$-dimensional torus,
see \cite{JJKLS2014}. In that paper, the almost sure dimension formula for
randomly placed
rectangles on $d$-dimensional torus is derived under a monotonicity assumption
on side lengths of the rectangles. The formula is as in \eqref{s2}
with $l_n$ replaced by the singular value function of a rectangle. Persson
\cite{Per2015} proved an almost sure lower bound for randomly placed open sets.
For rectangles, this lower bound equals the value obtained in \cite{JJKLS2014},
proving that the monotonicity assumption made in that paper is not needed.
Finally, in \cite{FJJS2018} Feng et al. derived an almost sure dimension formula
for general Lebesgue measurable generating sets, which completed the study in
the case where the randomness is determined by a measure having an absolutely
continuous component with respect to the Lebesgue measure.

There are  natural ways to continue the study of random covering sets:
one may replace the underlying space, which is a Riemann manifold in
\cite{FJJS2018}, by a metric space, one may consider singular generating
measures or one may study uniform coverings. The first line has been pursued in
\cite{JJKLSX2017,EJJS2018,EJJ2020,HL2021,Myl2021,E-B2022}. For uniform
coverings, see \cite{KLP2023}. For singular generating measures, the dimension
formula will depend both on the generating sets and the structure of the
generating measure. Seuret \cite{Seu2018} studied the case of balls with radii
$(n^{-\alpha})_{n=1}^\infty$ distributed according to a
Gibbs measure on the symbolic space and proved that the dimension formula
depends on $\alpha$ and the multifractal spectrum of the Gibbs measure. Fan et
al. \cite{FST2013} and Liao and Seuret \cite{LS2013} proved a similar result
for balls with radii $(n^{-\alpha})_{n=1}^\infty$ distributed along an orbit of an
expanding Markov map on the unit circle.

In \cite{EP2018}, Ekstr\"om and Persson studied balls with radii
$(n^{-\alpha})_{n=1}^\infty$ placed randomly according to a general measure on
$\mathbb R^d$. They proved that the Hausdorff dimension of the random covering
set is almost surely bounded from below by $\frac 1{\alpha}-\delta$
provided the upper Hausdorff dimension of the measure is larger than
$\frac 1{\alpha}$. Here $\delta$ is the essential infimum of the differences
of the upper and lower local dimensions in certain range, see
Theorem~\ref{EPthm}. For general $\alpha$, they gave almost sure
lower and upper bounds for the dimension depending on
the fine and coarse multifractal spectra of the measure and stated a
conjecture that the almost sure
dimension is equal to the value of the increasing 1-Lipschitz hull of the
fine multifractal spectrum at the point $\frac 1{\alpha}$. In particular, in the
case of Theorem~\ref{EPthm}, this value is equal to $\frac 1{\alpha}$.

In this paper, we prove that the Ekström-Persson conjecture is true in the
range where $\frac 1{\alpha}$ is at most the upper Hausdorff dimension of the
measure (Theorem~\ref{theorem:nminusalpha}). We also derive a lower bound for
the almost sure dimension which improves the one in Theorem~\ref{EPthm} and is
valid for all sequences of radii
(Theorem~\ref{theorem:lowerboundforgeneralballs}). We construct an example
showing that the bounds in Theorem~\ref{theorem:lowerboundforgeneralballs} are
sharp and that the
natural extension of the Ekström-Persson conjecture is not true for general
sequences of radii (Example~\ref{example:lowerboundmaybesharp}).
Example~\ref{example:dimensionalwayss_2} demonstrates that the
dimension may be independent of the difference between the lower and upper local
dimensions. Finally,
combining Examples~\ref{example:dimensionalwayss_2} and
\ref{example:lowerboundmaybesharp}, we conclude that there is no dimension
formula for general sequences of radii $\underline r$ which depends only on
$s_2(\underline r)$ and the lower and upper local dimensions of the generating
measure.

The paper is organised as follows. In Section~\ref{notation}, we
introduce some notation and state our main results. In Section~\ref{lemmas}, we
state some preliminary lemmas needed in later sections.
Section~\ref{section:proofofthmlowerboundenergy} is devoted to the proof of our
main technical tool (Theorem~\ref{theorem:generallowerboundenergy})
from which our main theorems (Theorem~\ref{theorem:nminusalpha} and
Theorem~\ref{theorem:lowerboundforgeneralballs})
follow. In Section~\ref{proofofEP}, we prove
Theorem~\ref{theorem:nminusalpha} and, in Section~\ref{proofofgeneral}, we prove
Theorem~\ref{theorem:lowerboundforgeneralballs}. Finally, in
Section~\ref{examples}, we construct examples
(Example~\ref{example:dimensionalwayss_2} and
Example~\ref{example:lowerboundmaybesharp}) showing the sharpness of our
results.

\section{Notation and Results}\label{notation}

We denote by $\mathcal{P}(\R^d)$ the space of compactly supported Borel
probability measures on $\R^d$ and denote by $\spt\mu$ the support of a measure
$\mu \in \mc{P}(\R^d)$. We write $B(x,r)$ for the open ball with centre
$x\in \R^d$ and radius $r>0$. We note that our results remain unchanged if all
the open balls are replaced by closed balls. We will use the notation
$\underline{r}=(r_k)_{k=1}^{\infty}$ for sequences of positive numbers.

For a measure $\mu \in \mathcal{P}(\R^d)$, the lower and upper local dimensions
of $\mu$ at a point $x$ are defined by
\[
\dimlocl \mu(x):=\liminf_{r\to 0}\frac{\log \mu(B(x,r))}{\log r}
\,\text{ and }\,
\dimlocu \mu(x):=\limsup_{r\to 0}\frac{\log \mu(B(x,r))}{\log r}.
\]
If these notions agree, we write $\dimloc \mu(x)$ for the common value. The
lower and upper Hausdorff dimensions of $\mu$ are defined by 
\[
\dimhl \mu:=\essinf_{x\sim \mu}\dimlocl \mu(x)
\,\text{ and }\,
\dimhu \mu:=\esssup_{x\sim \mu}\dimlocl \mu(x).
\]
Again, if these notions agree, we denote the common value by $\dimh \mu$.

\begin{definition}\label{definition:randomcoveringset}
Let $\mu\in\mathcal{P}(\R^d)$ and consider the probability space
$(\Omega,\mathbb{P})$, where $\Omega:=\spt\mu^{\N}$ and $\mathbb{P}:=\mu^{\N}$.
Let $\underline{r}=(r_k)_{k=1}^{\infty}$ be a sequence of positive numbers. Given
$\omega \in \Omega$, the \emph{random covering set} generated by the sequence
$\underline{r}$ is the limsup set
\[
E_{\underline{r}}(\omega)\coloneqq\limsup_{k\to\infty}B(\omega_{k},r_k)
  =\bigcap_{k=1}^{\infty}\bigcup_{n=k}^{\infty}B(\omega_n,r_n).
\]
In the special case $r_k=k^{-\alpha}$ for all $k\in \N$ and for some $\alpha>0$,
we write $E_{\alpha}(\omega)$ for the corresponding limsup set, that
is,
\[
E_{\alpha}(\omega):=\limsup_{k\to\infty}B(\omega_{k},k^{-\alpha}).
\]
\end{definition}

By Kolmogorov's zero-one law, the quantity
$\dim_{\mathrm{H}}E_{\underline{r}}(\omega)$ is constant $\mathbb{P}$-almost surely
since $\{\omega\in\Omega\mid\dim_{\mathrm{H}}E_{\underline{r}}(\omega)\le\beta\}$ is
a tail event for all $\beta\ge 0$ (see e.g. \cite[Lemma 3.1]{JJKLSX2017}). Let
$f_{\mu}(\underline{r})$ denote this almost sure value of
$\dim_{\mathrm{H}}E_{\underline{r}}(\omega)$. We also write $f_{\mu}(\alpha)$ for the
almost sure value of $\dimh E_{\alpha}(\omega)$. For a sequence $\underline{r}$
of positive numbers, let 
\[
s_1(\underline{r}):=\liminf_{k\to \infty}\frac{\log k}{-\log r_k}
\,\text{ and }\,
s_3(\underline{r}):=\limsup_{k\to \infty}\frac{\log k}{-\log r_k}.
\]
We always have the inequalities 
\[
s_1(\underline{r})\leq s_2(\underline{r})\leq s_3(\underline{r}),
\]
where $s_2(\underline{r})$ is as in (\ref{s2}). Furthermore, if $\underline{r}$
is decreasing, then $s_2(\underline{r})= s_3(\underline{r})$. This
follows from the well-known fact (due to Abel) that $\sum_{n=1}^\infty a_n<\infty$
implies $\lim_{n\to\infty}na_n=0$ if $(a_n)_{n=1}^\infty$ is decreasing.
Since the centres $\omega_k$ are independent identically distributed random
variables, the quantity $f_{\mu}(\underline{r})$ remains unchanged if the
sequence $\underline{r}$ is reordered. Thus, we may always assume that
$\underline{r}$ is decreasing. The quantity $s_2(\underline{r})$ is an upper
bound for $\dimh E_{\underline{r}}(\omega)$ for any realisation $\omega$. This can
easily  be seen by observing that 
\[
E_{\underline{r}}(\omega)\subseteq \bigcup_{k=N}^{\infty}B(\omega_k,r_k)
\]
for every $N\in \N$. In the special case where
$\underline{r}=(k^{-\alpha})_{n=1}^\infty$, we have that 
\begin{equation}\label{equation:s_1=s_3fornminusalpha}
    s_1(\underline{r})=s_3(\underline{r})=\frac{1}{\alpha}
\end{equation}
and, thus, $\frac{1}{\alpha}$ is always an upper bound for $f_{\mu}(\alpha)$. 
The almost sure dimension $f_{\mu}(\alpha)$ has been studied by Ekström and
Persson in \cite[Theorem 2.1]{EP2018}. They proved the following
result.

\begin{theorem}\label{EPthm}
Let $\mu\in \mc{P}(\R^d)$ and $\alpha>0$. If $\frac{1}{\alpha}<\dimhu\mu$,
then
\[
f_{\mu}(\alpha)\geq \frac{1}{\alpha}-\delta,
\]
where
\[
\delta:=\essinf_{\substack{x\sim\mu,\\ \dimlocl \mu(x)>1/\alpha}}
  \left( \dimlocu \mu(x)-\dimlocl \mu(x) \right).
\]
\end{theorem}

Ekström and Persson also obtained lower and upper bounds for $f_{\mu}(\alpha)$
when $\frac{1}{\alpha}>\dimhu \mu$, and they conjectured that the equality
\begin{equation}\label{conjecture1}
    f_{\mu}(\alpha)=\overline{F}_{\mu}\Bigl(\frac{1}{\alpha}\Bigr)
\end{equation}
always holds, where $F_{\mu}$ is the fine multifractal spectrum defined by
\[
F_{\mu}(s):=\dimh \{ x\in \spt \mu\mid \dimlocl \mu(x)\leq s\}
\]
and $\overline{F}_{\mu}$ denotes the increasing $1$-Lipschitz hull of $F_{\mu}$,
that is,
\[
\overline{F}_{\mu}(s):=\inf \{ h(s)\mid h\geq F_{\mu} \text{ is increasing and }
1\text{-Lipschitz continuous} \}.
\]
Our first main result is the following theorem, which verifies
\eqref{conjecture1} in the case where $\frac{1}{\alpha}\leq \dimhu \mu$.

\begin{theorem}\label{theorem:nminusalpha}
Let $\mu\in \mc{P}(\R^d)$ and $\alpha>0$. If $\frac{1}{\alpha}\le\dimhu\mu$,
then 
\[
f_{\mu}(\alpha)=\frac{1}{\alpha}.
\]
\end{theorem}

Since $E_\beta(\omega)\subset E_\alpha(\omega)\subset\spt\mu$ for $\beta>\alpha$,
we immediately obtain the following corollary.

\begin{corollary}
Let $\mu\in \mc{P}(\R^d)$ be such that $\dimh\spt\mu=\dimhu\mu$.
Then for every $\alpha>0$, 
\[
f_{\mu}(\alpha)=\min\{ \frac{1}{\alpha}, \dimh\spt\mu\}.
\]
\end{corollary}

Our second main theorem provides bounds for $f_{\mu}(\underline{r})$ for general
sequences $\underline{r}$.

\begin{theorem}\label{theorem:lowerboundforgeneralballs}
Let $\mu\in \mathcal{P}(\R^d)$ and let
$\underline{r}\coloneqq (r_k)_{k=1}^\infty$ be a sequence of positive
numbers tending to $0$. If $s_2(\underline{r})<\dimhu\mu$, then 
\[
s_2(\underline{r})\overline{\delta}\leq f_{\mu}(\underline{r})
  \leq s_2(\underline{r}),
\]
where
\[
\overline{\delta}\coloneqq \esssup_{\substack{x\sim\mu,\\
    \dimlocl \mu(x)>s_2(\underline{r})}}\frac{\dimlocl \mu(x)}{\dimlocu \mu(x)}.
\]
\end{theorem}

\begin{remark}
a) One may apply the methods used in \cite{EP2018} to prove
Theorem~\ref{EPthm} also for a general sequence $\underline r$. In order to do
this, one has to replace $\frac 1\alpha$ by $s_1(\underline r)$ in
some places and by $s_3(\underline r)$ in other places, and this leads
to the lower bound
\[
f_{\mu}(\underline{r})\geq s_1(\underline{r})-\hat\delta,
\]
where
\[
\hat\delta:=\essinf_{\substack{x\sim\mu,\\ \dimlocl \mu(x)>s_3(\underline r)}}
  \left( \dimlocu \mu(x)-\dimlocl \mu(x) \right).
\]
Recall that, for a general decreasing sequence $\underline r$, the
strict inequality $s_1(\underline r)<s_2(\underline r)$ is possible,
whilst $s_2(\underline r)=s_3(\underline r)$ for all decreasing sequences. 
        
b) The lower bound $s_2(\underline{r})\overline{\delta}$ obtained in
Theorem~\ref{theorem:lowerboundforgeneralballs} is always
positive and larger than the quantity $s_2(\underline{r})-\hat\delta$.
Indeed, fix $\mu \in \mc{P}(\R^{d})$ and a sequence $\underline{r}$ such that
$s_2\coloneqq s_2(\underline{r})<\overline{\dim}_{\mathrm{H}}\,\mu$.
By the definition of $\overline{\delta}$, we have for $\mu$-almost
every $x$ satisfying the property $\dimlocl\mu(x)>s_2$ that
\begin{align*}
s_2\overline{\delta}
  &\geq s_2 \frac{\dimlocl\mu(x)}{\dimlocu \mu(x)}= s_2-
    \frac{s_2}{\dimlocu \mu(x)}\left( \dimlocu \mu(x)-\dimlocl\mu(x)\right)\\
  &\geq s_2-\left( \dimlocu \mu(x)-\dimlocl\mu(x)\right),
\end{align*}
where the last inequality follows from the fact that 
$s_2<\dimlocl\mu(x)\leq \dimlocu \mu(x)$. Since the above inequality holds true
for $\mu$-almost every $x$ such that $\dimlocl\mu(x)>s_2$, we have (by the
definition of $\hat\delta$) that 
$s_2\overline{\delta}\geq s_2-\hat\delta$.
\end{remark}

By noting that $\overline{\delta}=1$ in
Theorem~\ref{theorem:lowerboundforgeneralballs}, if $\dimloc \mu(x)$ exists
and is larger than $s_2(\underline r)$ in a
set of positive measure, we obtain the following corollary.

\begin{corollary}
Suppose that $\mu\in \mc{P}(\R^d)$ is such that the local dimension
$\dimloc \mu(x)$ exists in a set of positive $\mu$-measure. Then, for any
sequence $\underline{r}$ satisfying
$s_2(\underline{r})<\esssup_{x\sim\mu}\dimloc\mu(x)$, we have that
\[
f_{\mu}(\underline{r})=s_2(\underline{r}). 
\]
\end{corollary}

In the proof of our main technical result
(Theorem~\ref{theorem:generallowerboundenergy}), we will make use of potential
theoretic arguments. The $t$-potential of a measure $\mu\in \mathcal{P}(\R^d)$
at a point $x$ is defined by 
\[
\phi^{t}_{\mu}(x):=\int |x-y|^{-t}\,d\mu(y)
\]
and the $t$-energy of $\mu$ is 
\[
I_t(\mu):=\int\phi^{t}_{\mu}(x)\,d\mu(x)=\int\int|x-y|^{-t}\,d\mu(y)\,d\mu(x).
\]
The mutual $t$-energy of two measures $\mu,\nu\in\mathcal{P}(\R^d)$
is defined by
\[
J_t(\mu,\nu):=\int\int|x-y|^{-t}\,d\mu(y)\,d\nu(x).
\]
We also define the $t$-capacity of a set $A\subseteq \R^d$ by 
\[
\capa_t(A):=\sup\{ I_t(\mu)^{-1}\mid \mu\in \mathcal{P}(\R^d),\spt\mu\subset A\}.
\]
It is well known that if $\capa_t(A)>0$, then $\dimh A\geq t$ (see e.g.
\cite[Chapter 8]{Mat95}). A measure $\mu\in \mathcal{P}(\R^d)$ is
$(C,s)$-Frostman if $\mu(B(x,r))\leq Cr^{s}$ for every $x\in \R^d$ and $r>0$. If
there is no need to emphasise the constant $C$, we will just say that $\mu$ is
$s$-Frostman. It is well known that $I_t(\mu)<\infty$
for all $t<s$ provided $\mu$ is $s$-Frostman (see \cite[p. 109]{Mat95}). 
Given $\mu\in \mathcal{P}(\R^d)$ and $A\subseteq \R^d$, we write $\mu_{A}$ for
the normalised restriction of $\mu$ to $A$, that is,
\[
\mu_{A}(B)=\mu(A)^{-1}\mu_{|_{A}}(B)=\mu(A)^{-1}\mu(A\cap B)
\]
for $B\subseteq \R^d$ with the interpretation that $\mu_{A}$ is the zero measure
whenever $\mu_{|_A}$ is.

The proofs of Theorems~\ref{theorem:nminusalpha} and
\ref{theorem:lowerboundforgeneralballs} are based on the following theorem,
proved in Section~\ref{section:proofofthmlowerboundenergy}.

\begin{theorem}\label{theorem:generallowerboundenergy}
Let  $\mu\in \mathcal{P}(\R^d)$ be $(C,s)$-Frostman. Let
$(r_k)_{k=1}^{\infty}$ be a decreasing sequence of positive numbers tending to
zero. Let $0<t<s$. Suppose that there exists a bounded sequence of
positive numbers $(b_k)_{k=1}^{\infty}$ such that
\begin{equation}\label{equation:energyassumptionforlowerbound}
\mu\Bigl(\Bigl\{ x\in\spt\mu\mid\sum_{k=1}^\infty\chi_{A_k^{t}}(x)b_k
  =\infty\Bigr\}\Bigr)=1,
\end{equation}
where
$A_k^{t}:=\left\{x\in\spt\mu\mid I_t(\mu_{B(x,r_k)})\leq b_k^{-1}\right\}$
and $\chi_B$ is the characteristic function of a set $B$.  Then
$f_{\mu}(\underline{r})\geq t$.
\end{theorem}

\section{Preliminary lemmas}\label{lemmas}

In this section, we state and prove some lemmas needed in the proof
of Theorem~\ref{theorem:generallowerboundenergy}.

\begin{lemma}\label{lemma:mattila1.15}
Let $\mu\in \mathcal{P}(\R^d)$ and let $\varphi\geq 0$ be a Borel function. Then
\[
\int \varphi \, d\mu= \int_{0}^{\infty}\mu(\{z\mid \varphi(z)>t \})\,dt.
\]
\end{lemma}

\begin{proof} See \cite[Theorem 1.15]{Mat95}. In that theorem, it
is assumed that $\varphi(z)\ge t$, but the same proof works also in the case
$\varphi(z)>t$.    
\end{proof}

\begin{lemma}\label{lemma:energyupperboundfrostmanmeasures}
Assume that $\mu\in \mathcal{P}(\R^d)$ is $(C,s)$-Frostman. 
Let $0<t<s$, $z_0\in\spt\mu$ and $r_0>0$. Then the following
inequalities are valid. For every $y\in \R^d$,
\begin{align*}
\phi^{t}_{\mu_{|_{B(z_0,r_0)}}}(y)&\leq\frac{C^{\frac ts}s}{s-t}\mu(B(z_0,r_0))^{1-\frac ts}
   \leq\frac{Cs}{s-t} r_0^{s-t}\text{ and}\\
   I_{t}(\mu_{B(z_0,r_0)})&\leq\frac{C^{\frac ts}s}{s-t}\mu(B(z_0,r_0))^{-\frac ts}.
\end{align*}
\end{lemma}

\begin{proof}
The second inequality on the first line is trivial, since $\mu$ is
$(C,s)$-Frostman. To prove the first inequality, write
$B_0\coloneqq B(z_0,r_0)$ and let 
\[
\rho_0:=\sup\{0<\rho\leq r_0\mid C\rho^{s}\leq \mu(B_0)\}.
\]
Since $0<\mu(B_0)\leq Cr_0^{s}$ and the map $\rho\mapsto C\rho^{s}$
is continuous, $\rho_0\in\mathopen]0,r_0\mathclose]$ and
$\rho_0=\left(C^{-1}\mu(B_0)\right)^{\frac 1s}$. Fix $y\in \R^d$. By
Lemma~\ref{lemma:mattila1.15} and a change of variables,
\begin{align*}
\phi^{t}_{\mu_{|_{B_0}}}(y)&=\int_{B_0}|x-y|^{-t}\,d\mu(x)
    =\int_{0}^{\infty}\mu(\{x\in B_0\mid |x-y|^{-t}>\gamma\})\, d\gamma\\
 &=\int_{0}^{\infty}\mu(B_0\cap B(y,\gamma^{-\frac 1t}))\, d\gamma
    = t\int_{0}^{\infty}\mu(B_0\cap B(y,a))a^{-(t+1)}\,da\\
 &\leq t\left(\int_{0}^{\rho_0}Ca^{s-t-1}\,da+\int_{\rho_0}^{\infty}\mu(B_0)a^{-(t+1)}
    \,da \right)
    \le t\left(\frac{C}{s-t}\rho_{0}^{s-t}+\mu(B_0)\frac{\rho_{0}^{-t}}{t}\right)\\
 &=\frac{Ct}{s-t}\left(C^{-1}\mu(B_0) \right)^{\frac{s-t}{s}}+\mu(B_0)(\left(C^{-1}
    \mu(B_0) \right)^{\frac 1s})^{-t}\\
 &=\left(\frac{Ct}{s-t}C^{\frac{t-s}{s}}+C^{\frac ts} \right)\mu(B_0)^{1-\frac ts}
    =\frac{C^{\frac ts}s}{s-t}\mu(B_0)^{1-\frac ts},
\end{align*}
which proves the first inequality. The last inequality follows by
multiplying this point wise estimate by $\mu(B_0)^{-1}$ and integrating with
respect to the probability measure $\mu_{B_0}$. 
\end{proof}

\begin{lemma}\label{lemma:liminfatmostexpectation}
Let $(\xi_n)_{n=1}^\infty$ be a sequence of independent random variables. Then 
\[
\liminf_{n\to \infty} \xi_n\leq \liminf_{n\to \infty} \mathbb{E}(\xi_n)
\]
almost surely.
\end{lemma}

\begin{proof}
\cite[Lemma 3.6]{EJJ2020}.
\end{proof}

\begin{lemma}\label{lemma:conditiononalmostfullsets}
Let $X\subset\R^d$ be compact and let $\nu\in \mc{P}(X)$. Suppose that
$(A_n)_{n=1}^\infty$ is a sequence of Borel subsets of $X$ with
$\lim_{n\to\infty}\nu(A_n)=1$. Then $\lim_{n\to \infty}\nu_{A_n}=\nu$ in
the weak-$\ast$ topology and
\[
\lim_{n\to \infty}I_t(\nu_{A_n})=I_t(\nu)
\]
for every $t>0$.
\end{lemma}

\begin{proof}
\cite[Lemma 3.5]{FJJS2018}.
\end{proof}

We finish this section with a simplified version of a lemma proved in
\cite{EP2018}.  

\begin{lemma}\label{lemma:ekstromperssonboundeddensities}
Let $\mu\ll\nu$ be Borel probability measures on a separable metric
space $X$ such that the density $\frac{d\mu}{d\nu}$ is bounded. Then
$f_{\mu}(\underline{r})\leq f_{\nu}(\underline{r})$ for every sequence
$\underline{r}$.
\end{lemma}

\begin{proof}
\cite[Lemma 9.3]{EP2018}.
\end{proof}

\section{Proof of Theorem \ref{theorem:generallowerboundenergy}}
\label{section:proofofthmlowerboundenergy}

In the proof of Theorem \ref{theorem:generallowerboundenergy}, we will make use
of the following deterministic result.

\begin{lemma}\label{lemma:deterministiclowerbound}
Let $\nu$ be a finite Borel measure on a compact metric space $X$ and let
$(\varphi_n)_{n=1}^\infty$ be a sequence of non-negative continuous functions on
$X$ such that $\lim_{n\to \infty}\varphi_{n}\,d\nu=\nu$ in the weak-$\ast$ topology
and $\liminf_{n\to\infty}I_t(\varphi_{n}\rho\,d\nu)\leq I_t(\rho\,d\nu)$ whenever
$\rho$ is a product of finitely many of the functions $\{\varphi_n\}_{n\in\N}$.
Then for every $t>0$,
\[
\capa_t\bigl(\spt\nu\cap\limsup_{n\to\infty}(\spt\varphi_n)\bigr)
  \geq \frac{\nu(X)^{2}}{I_t(\nu)}.
\]
\end{lemma}

\begin{proof}
\cite[Lemma 1.4]{EJJ2020}.
\end{proof}

We are now ready to prove Theorem \ref{theorem:generallowerboundenergy}.

\begin{proof}[Proof of Theorem \ref{theorem:generallowerboundenergy}]
  \renewcommand{\qedsymbol}{}
Fix $0<t<s$ and a sequence $(b_k)_{k=1}^\infty$ such that
\eqref{equation:energyassumptionforlowerbound} holds. Our aim is to show that
$f_{\mu}(\underline{r})\geq t$. Since $t>0$ is fixed, we will denote the sets 
$A_k^{t}$ from the statement simply by $A_k$. By our assumption,
\[
\sum_{k=1}^\infty b_k\chi_{A_k}(x)=\infty
\]
for $\mu$-almost every $x$. Hence, we can choose sequences $(M_n)_{n=1}^\infty$
and $(N_n)_{n=1}^\infty$ of natural numbers such that $M_n<N_n<M_{n+1}$ for every
$n\in \N$, 
\begin{equation}\label{Fn}
r_{M_n}<\frac{4^{-n}}{2}\text{ and }\mu\left( F_n \right)\geq 1-2^{-n},
\end{equation}  
where 
\[
F_n\coloneqq\Bigl\{x\in\spt\mu\mid\sum_{k=M_n}^{N_n}\chi_{A_k}(x)b_k
\geq 2^{n}\Bigr\}.
\]
Clearly, $F_n$ is a Borel set. We now establish some notation. Set 
\[
\Sigma_n\coloneqq \{\textbf{i}=(i_{M_n},\ldots,i_{N_n})\in \{0,1 \}^{N_n-M_n+1} \}.
\]
For $\textbf{i}=(i_{M_n},\ldots,i_{N_n})\in \Sigma_n$, let 
\[
A_{n,\textbf{i}}\coloneqq\biggl(\bigcap_{\substack{j=M_n\\ i_j=1}}^{N_n}A_j\biggr)
\setminus\biggl( \bigcup_{\substack{j=M_n\\ i_j=0}}^{N_n}A_j \biggr) \text{ and }
c_{n,\textbf{i}}\coloneqq \sum_{\substack{j=M_n\\ i_j=1}}^{N_n}b_j.
\]

Set $\Sigma_{n}'\coloneqq\{\textbf{i}\in\Sigma_n\mid c_{n,\textbf{i}}\geq 2^{n}\}$.
Note that 
\[
\bigcup_{k=M_n}^{N_n}A_k=\bigcup_{\textbf{i}\in \Sigma_n}A_{n,\textbf{i}}
\]
with the right hand side union disjoint, and 
\begin{equation}\label{Fncover}
F_n=\bigcup_{\textbf{i}\in \Sigma_{n}'}A_{n,\textbf{i}}.
\end{equation}  

We will now consider the Borel probability measures
\begin{equation}\label{mun}
\mu_{n}\coloneqq \mu(F_n)^{-1}\sum_{\textbf{i}\in \Sigma_{n}'}\mu(A_{n,\textbf{i}})
a_{n,\textbf{i}}\sum_{j=M_n}^{N_n}b_{j}\,\mu_{|_{A_j\cap A_{n,\textbf{i}}}},
\end{equation}
where
\begin{equation}\label{ani}
a_{n,\textbf{i}}\coloneqq\Bigl(\sum_{j=M_n}^{N_n}b_{j}\mu(A_j\cap A_{n,\textbf{i}})
\Bigr)^{-1}
\end{equation}
provided that $\mu(A_{n,\textbf{i}})>0$. Observe that 
\[
\sum_{j=M_n}^{N_n}b_{j}\,\mu(A_j\cap A_{n,\textbf{i}})=
\int_{A_{n,\textbf{i}}}\sum_{j=M_n}^{N_n}b_{j}\,\chi_{A_j}(x)\,d\mu(x)
=c_{n,\textbf{i}}\,\mu(A_{n,\textbf{i}}),
\]
hence 
\begin{equation}\label{acrelation}
a_{n,\textbf{i}}=(c_{n,\textbf{i}}\mu(A_{n,\textbf{i}}))^{-1}
\end{equation}
for $\mu(A_{n,\textbf{i}})>0$. Thus, for any Borel set $B\subset\R^d$, we have that
\begin{align*}
\mu_n(B)&=\mu(F_n)^{-1}\sum_{\textbf{i}\in \Sigma_{n}'}\mu(A_{n,\textbf{i}})a_{n,\textbf{i}}
   \sum_{j=M_n}^{N_n}b_{j}\,\mu(A_j\cap A_{n,\textbf{i}}\cap B)\\
 &=\mu(F_n)^{-1}\sum_{\textbf{i}\in \Sigma_{n}'}\mu(A_{n,\textbf{i}})a_{n,\textbf{i}}
   \int_{A_{n,\textbf{i}}\cap B}\sum_{j=M_n}^{N_n}b_{j}\,\chi_{A_j}(x)\, d\mu(x)\\
 &=\mu(F_n)^{-1}\sum_{\textbf{i}\in \Sigma_{n}'}\mu(A_{n,\textbf{i}})a_{n,\textbf{i}}\,
   c_{n,\textbf{i}}\,\mu(A_{n,\textbf{i}}\cap B)\\
 &=\mu(F_n)^{-1}\sum_{\textbf{i}\in \Sigma_{n}'}\mu(A_{n,\textbf{i}}\cap B)
   =\mu(F_n)^{-1}\mu(F_n\cap B).
\end{align*}
Hence
\begin{equation}\label{munismuFn}
\mu_{n}=\mu_{F_n}.
\end{equation}
Since $\lim_{n\to\infty}\mu(F_n)=1$,
Lemma~\ref{lemma:conditiononalmostfullsets} implies that
$\mu_n\xrightarrow[n\to\infty]{w\text{-}\ast}\mu$ and
$I_t(\mu_{n})\xrightarrow[n\to\infty]{} I_t(\mu)$ (recall that $\mu$ is
$s$-Frostman, hence $I_t(\mu)<\infty$).

For all $n\in\N$, define a Borel function $E_n\colon\R^d\to\R$ by 
\begin{equation}\label{En}
E_n(x)\coloneqq \mu(F_n)^{-1}\sum_{\textbf{i}\in \Sigma_{n}'}\mu(A_{n,\textbf{i}})
a_{n,\textbf{i}}\sum_{j=M_n}^{N_n}b_{j}\int_{A_j\cap A_{n,\textbf{i}}}
\frac{\chi_{B(x,r_k)}(z)}{\mu(B(z,r_k))}\, d\mu(z).
\end{equation}
Using the fact $\chi_{B(x,r)}(z)=\chi_{B(z,r)}(x)$ and Fubini's
theorem, we conclude that 
\[
\int E_n(x)\, d\mu(x)=1,
\]
hence Markov's inequality yields
\[
\mu(\{x\in\spt\mu\mid E_n(x)\geq n^2 \})\leq n^{-2}.
\]
Fix $N\in \N$ large enough so that $\mu(G)>0$, where 
\begin{equation}\label{G}
    G\coloneqq \bigcap_{n=N}^\infty\{x\in\spt\mu\mid E_n(x)< n^{2} \}.
\end{equation}
Clearly, $G$ is a Borel set. Such an $N$ exists since 
\[
\mu(\R^d\setminus G)\leq \sum_{n=N}^{\infty}n^{-2}.
\]
By the Lebesgue's density theorem, 
\[
\lim_{r\downarrow 0}\frac{\mu(G\cap B(x,r))}{\mu(B(x,r))}=1
\]
for $\mu$-almost every $x\in G$. Thus, we can choose a decreasing sequence
$(\vep_n)_{n=1}^{\infty}$ of positive numbers tending to $0$ such that
\begin{equation}\label{Hnprop}
\mu(H_n)\geq (1-\frac{1}{n})\mu(G),
\end{equation}
where
\begin{equation}\label{Hn}
H_n:=\Bigl\{x\in G\mid \frac{\mu(G\cap B(x,r))}{\mu(B(x,r))}\geq 1-\frac{1}{n}
\text{ for all } 0<r<\vep_n\Bigr\}. 
\end{equation}
For all $n\in\N$, we write 
\[
q_n\coloneqq\min\bigl\{n,\max\{k\in\N\mid\vep_k>r_{M_n}\}\bigr\}.
\]
Then
\begin{equation}\label{qnprop}
\lim_{n\to\infty}q_n=\infty\text{ and }r_k<\vep_{q_n}\text{ for all }
k\in\{ M_n,\ldots,N_n\}.
\end{equation}

Our goal is to construct, for almost every $\omega \in \Omega$, a measure
supported on $E_{\underline{r}}(\omega)$  having finite $t$-energy. To this end,
for each $\omega \in \Omega$ and $k\in\N$, write
$B_k^{\omega}\coloneqq B(\omega_k,r_k)$ and let $U_{k}^{\omega}\subset B_k^{\omega}$
be the smallest closed ball centred at $\omega_k$ such that
$\mu(G\cap U_k^{\omega})\geq \frac{\mu(G\cap B_k^{\omega})}{2}$. Since $B_k^{\omega}$
is open, $\dist(U_k^{\omega},\R^d\setminus B_k^{\omega})>0$. Let $r(\omega_k)$ be
the radius of $U_k^\omega$. Note that
$\lim_{\delta\to 0}\mu(G\cap B(y,\rho-\delta))=\mu(G\cap B(y,\rho))$ for all
$y\in\R^d$ and $\rho>0$. Fix $y\in\R^d$, $\vep>0$ and $k\in\N$. Let
$\lim_{i\to\infty}y_i=y$ and denote by $\overline B(z,\rho)$ the closed ball
centred at $z$ with radius $\rho$. There exists $\delta>0$ such that, for all
large $i\in\N$, we have by the definition of $r(y)$ that
\begin{align*}
\mu(G\cap\overline B(y_i,r(y)-2\vep))&\le\mu(G\cap\overline B(y,r(y)-\vep))
   <\frac{\mu(G\cap B(y,r_k-\delta))}2\\
 &\le\frac{\mu(G\cap B(y_i,r_k))}2.
\end{align*}  
Thus, $r(y_i)\ge r(y)-2\vep$ for all large $i\in\N$, which implies that $r$ is
lower semicontinuous. Therefore, $(\omega,x)\mapsto\chi_{B_k^{\omega}}(x)$ and
$(\omega,x)\mapsto\chi_{U_k^{\omega}}(x)=\chi_{\overline B(\omega_k,r(\omega_k))}(x)$
are Borel maps and we may choose a continuous function $\tilde{\psi}_{k}^{\omega}$
such that
\begin{equation}\label{tildepsi}
\chi_{U_k^{\omega}}\leq \tilde{\psi}_{k}^{\omega}\leq \chi_{B_k^{\omega}}
\end{equation}  
and $(\omega,x)\mapsto\tilde\psi_k^\omega(x)$ is a Borel map. For example, we may
interpolate linearly between the boundaries of $U_k^{\omega}$ and $B_k^{\omega}$.  
Finally, for all $k\in\N$ and $\omega\in\Omega$ such that
$\mu(G\cap B_k^{\omega})>0$, let 
\begin{equation}\label{equation:defnofpsi_k^omega}
\psi_{k}^{\omega}:=c_k^{\omega}\frac{\Tilde{\psi}_k^{\omega}}{\mu(G\cap B_k^{\omega})},
\end{equation}
where $c_{k}^{\omega}\leq 2$ is such that $\psi_{k}^{\omega}\, d\mu_{|_{G}}$ is a
Borel probability measure. Note that $\omega\mapsto c_k^\omega$ is a Borel
function.

We will now consider the Borel measures $\varphi_{n}^{\omega}\, d\mu_{|_{G}}$, where
\begin{equation}\label{varphinomega}
\varphi_{n}^{\omega}:=\mu(F_n)^{-1}\sum_{\textbf{i}\in \Sigma'_{n}}\mu(A_{n,\textbf{i}})
a_{n,\textbf{i}}\sum_{k=M_n}^{N_n}b_{k}\,\chi_{A_k\cap A_{n,\textbf{i}}\cap H_{q_n}}(\omega_k)
\psi_{k}^{\omega}.
\end{equation}  
We require the following lemma.
\end{proof}

\begin{sublemma}\label{sublemma}
Let
\[
\mu_{n}^{\omega}\coloneqq\mu(F_n)^{-1}\sum_{\bf{i}\in \Sigma'_{n}} \mu(A_{n,\bf{i}})
a_{n,\bf{i}}\sum_{k=M_n}^{N_n}b_{k}\,\chi_{A_k\cap A_{n,\bf{i}}\cap H_{q_n}}(\omega_k)
\delta_{\omega_k},
\]
where $\delta_{\omega_k}$ is the Dirac measure at $\omega_k$. Then
$\mu_{n}^{\omega}\xrightarrow[n\to\infty]{w\text{-}\ast}\mu_{|_{G}}$
almost surely.
\end{sublemma}

\begin{proof}
Let $h\in C(\R^d)$ with compact support. Since
$\lim_{n\to\infty}\mu(F_n)=1$ by \eqref{Fn} and $H_{q_n}\cap F_n\subset G$ by
\eqref{Hn}, we have that $\lim_{n\to\infty}\mu(H_{q_n}\cap F_n)=\mu(G)$ by
\eqref{Hnprop} and \eqref{qnprop}. Recalling that $\mu_n=\mu_{F_n}$ (see
\eqref{munismuFn}) and writing $\nu(h):=\int h\,d\nu$ for a Borel measure $\nu$,
we obtain
\begin{align*}
\mathbb{E}(\mu_{n}^{\omega}(h))&=\mu(F_n)^{-1}\sum_{\textbf{i}\in\Sigma'_{n}}
   \mu(A_{n,\textbf{i}})a_{n,\textbf{i}}\sum_{k=M_n}^{N_n}b_{k}\,
   \int_{A_k\cap A_{n,\textbf{i}}\cap H_{q_n}}h(\omega_k)\,d\mu(\omega_k)\\
 &=\int_{H_{q_n}}\,d\mu_n=\mu(F_n)^{-1}\int_{H_{q_n}}h\,d\mu_{|_{F_n}}
   =\mu(F_n)^{-1}\int h\,d\mu_{|_{H_{q_n}\cap F_n}}\\
 &\xrightarrow[n\to\infty]{}\mu_{|_G}(h).
\end{align*}
Set $b:=\sup_{k\in\N}b_k$. Recalling \eqref{Fncover}, disjointedness of
the sets $A_{n,\textbf{i}}$, \eqref{acrelation}, the definition of $\Sigma_n'$,
\eqref{mun} and \eqref{munismuFn},
we estimate the variance
\begin{align*}
&\mathrm{Var}(\mu_{n}^{\omega}(h))\\
&=\mathrm{Var}\biggl(\mu(F_n)^{-1}\sum_{k=M_n}^{N_n}b_{k}\,\sum_{\textbf{i}\in\Sigma'_{n}}
  \mu(A_{n,\textbf{i}})a_{n,\textbf{i}} \chi_{A_k\cap A_{n,\textbf{i}}\cap H_{q_n}}(\omega_k)
  h(\omega_k)\biggr)\\
&=\mu(F_n)^{-2}\sum_{k=M_n}^{N_n}b_{k}^{2}\,\mathrm{Var}\biggl(
  \sum_{\textbf{i}\in \Sigma'_{n}}\mu(A_{n,\textbf{i}})a_{n,\textbf{i}}
  \chi_{A_k\cap A_{n,\textbf{i}}\cap H_{q_n}}(\omega_k) h(\omega_k)\biggr)\\
&\leq \Vert h\Vert_{\infty}^{2}\mu(F_n)^{-2}\sum_{k=M_n}^{N_n}b_{k}^{2}\, \mathbb{E}
  \biggl(\Bigl(\sum_{\textbf{i}\in \Sigma'_{n}} \mu(A_{n,\textbf{i}})a_{n,\textbf{i}}
  \chi_{A_k\cap A_{n,\textbf{i}}}(\omega_k)\Bigr)^2\biggr)\\
&=\Vert h\Vert_{\infty}^{2}\mu(F_n)^{-2}\sum_{k=M_n}^{N_n}b_{k}^{2}\,
  \sum_{\textbf{j}\in \Sigma'_{n}}\int_{A_{n,\textbf{j}}}\biggl(\sum_{\textbf{i}\in \Sigma'_{n}}
  \mu(A_{n,\textbf{i}})a_{n,\textbf{i}} \chi_{A_k\cap A_{n,\textbf{i}}}(\omega_k)\biggr)^2\,
  d\mu(\omega_k)\\
&=\Vert h\Vert_{\infty}^{2}\mu(F_n)^{-2}\sum_{k=M_n}^{N_n}b_{k}^{2}\,
  \sum_{\textbf{j}\in \Sigma'_{n}}\int_{A_{n,\textbf{j}}}\mu(A_{n,\textbf{j}})^{2}
  a_{n,\textbf{j}}^{2} \chi_{A_k\cap A_{n,\textbf{j}}}(\omega_k)\, d\mu(\omega_k)\\
&=\Vert h\Vert_{\infty}^{2}\mu(F_n)^{-2}\sum_{\textbf{j}\in \Sigma'_{n}}
  \mu(A_{n,\textbf{j}})^{2}a_{n,\textbf{j}}^{2}\sum_{k=M_n}^{N_n}b_{k}^{2}\,
  \mu(A_k\cap A_{n,\textbf{j}})\\
&=\Vert h\Vert_{\infty}^{2}\mu(F_n)^{-2}\sum_{\textbf{j}\in \Sigma'_{n}}\mu(A_{n,\textbf{j}})
  c_{n,\textbf{j}}^{-1}a_{n,\textbf{j}}\sum_{k=M_n}^{N_n}b_{k}^{2}\,
  \mu(A_k\cap A_{n,\textbf{j}})\\
&\leq\Vert h\Vert_{\infty}^{2}b2^{-n}\mu(F_n)^{-2}
  \sum_{\textbf{j}\in \Sigma'_{n}}\mu(A_{n,\textbf{j}})a_{n,\textbf{j}}\sum_{k=M_n}^{N_n}b_{k}\,
  \mu(A_k\cap A_{n,\textbf{j}})\\
&=\Vert h\Vert_{\infty}^{2}b2^{-n}\mu(F_n)^{-1}.
\end{align*}
Thus 
\[
\sum_{n=1}^\infty\mathrm{Var}(\mu_{n}^{\omega}(h))<\infty.
\]
Given $\delta>0$, we have by Chebyshev's inequality that 
\[
\sum_{n=1}^\infty\Pb\left\{|\mu_{n}^{\omega}(h)-\mu_{|_G}(h) |\geq\delta\right\}
\le\sum_{n=1}^\infty\delta^{-2}\mathrm{Var}(\mu_{n}^{\omega}(h))<\infty.
\]
Hence, by Borel-Cantelli lemma, 
\[
\limsup_{n\to \infty}|\mu_{n}^{\omega}(h)-\mu_{|_G}(h) |\leq \delta
\]
almost surely and, thus,
$\lim_{n\to\infty}\mu_{n}^{\omega}(h)=\mu_{|_G}(h)$ almost surely. Since
$C_0(\R^d)$ is separable,
$\mu_{n}^{\omega}\xrightarrow[n\to\infty]{w\text{-}\ast}\mu_{|_{G}}$ almost surely.
This completes the proof of Sublemma~\ref{sublemma}.  
\end{proof}

\begin{proof}[Proof of Theorem \ref{theorem:generallowerboundenergy} continued]
Since $\mu_{n}^{\omega}\xrightarrow[n\to\infty]{w\text{-}\ast}\mu_{|_G}$
almost surely, $\lim_{k\to\infty}r_k=0$ and $\psi_{k}^{\omega}\, d\mu_{|_{G}}$ is a
probability measure for $\omega_k\in H_{q_n}$ (see
\eqref{equation:defnofpsi_k^omega}), we have that
$\varphi_{n}^{\omega}\,d\mu_{|_{G}}\xrightarrow[n\to\infty]{w\text{-}\ast}\mu_{|_G}$
almost surely. Let $\rho\in C_0(\R^d)$ be non-negative. In order to
apply Lemma~\ref{lemma:deterministiclowerbound}, we will estimate the
$t$-energies of the random measures $\varphi_{n}^{\omega}\rho\,d\mu_{|_{G}}$. 

Observe first that 
\begin{equation}
\mathbb{E}\left( I_t(\varphi_{n}^{\omega}\rho\,d\mu_{|_{G}})\right)=S_1+S_2,
\end{equation}
where
\begin{align*}
S_1&:=\mu(F_n)^{-2}\sum_{k=M_n}^{N_n}b_{k}^{2}\,\sum_{\textbf{i},\textbf{j}\in \Sigma_{n}'}
   \mu(A_{n,\textbf{i}})\mu(A_{n,\textbf{j}})a_{n,\textbf{i}}a_{n,\textbf{j}}\\
 &\phantom{:=\mu(F_n)^{-2}}\times \mathbb{E}\left(
   \chi_{A_k\cap A_{n,\textbf{i}}\cap H_{q_n}}(\omega_k)\chi_{A_k\cap A_{n,\textbf{j}}\cap H_{q_n}}
   (\omega_k) I_t(\psi_{k}^{\omega}\rho\,d\mu_{|_{G}})\right)\text{ and}\\
S_2&:=\mu(F_n)^{-2}\sum_{\substack{l,m=M_n\\ l\neq m}}^{N_n}b_{l}\,b_{m}\,
   \sum_{\textbf{i},\textbf{j}\in \Sigma_{n}'}\mu(A_{n,\textbf{i}})\mu(A_{n,\textbf{j}})
   a_{n,\textbf{i}}a_{n,\textbf{j}}\\
 &\phantom{:=\mu(F_n)^{-2}}\times\mathbb{E}\left(
   \chi_{A_l\cap A_{n,\textbf{i}}\cap H_{q_n}}(\omega_l)\chi_{A_m\cap A_{n,\textbf{j}}\cap H_{q_n}}
   (\omega_m) J_t(\psi_{l}^{\omega}\rho\, d\mu_{|_{G}},\psi_{m}^{\omega}\rho\,
   d\mu_{|_{G}})\right).\\
\end{align*}
We first estimate $S_1$. By \eqref{equation:defnofpsi_k^omega}, we have for
$\omega_k\in A_k\cap H_{q_n}$ that
\[
I_t(\psi_{k}^{\omega}\rho\, d\mu_{|_{G}})\leq 4\Vert \rho\Vert_{\infty}^{2}
I_t\Bigl(\frac{\chi_{B_k^{\omega}}}{\mu(G\cap B_k^{\omega})}\, d\mu_{|_G}\Bigr)
\leq 4\Vert\rho\Vert_{\infty}^{2}\Bigl(1-\frac{1}{q_n}\Bigr)^{-2} b_k^{-1}.
\]
Thus, using the fact that the sets $A_{n,\textbf{i}}$ are disjoint,
\eqref{acrelation}, \eqref{ani} and \eqref{Fncover}, we obtain
\begin{align*}
S_1&\leq 4\Vert \rho\Vert_{\infty}^{2}\Bigl(1-\frac{1}{q_n}\Bigr)^{-2}\mu(F_n)^{-2}
   \sum_{k=M_n}^{N_n}b_{k}^{2}\,\sum_{\textbf{i},\textbf{j}\in \Sigma_{n}'}\mu(A_{n,\textbf{i}})
   \mu(A_{n,\textbf{j}})a_{n,\textbf{i}}a_{n,\textbf{j}}\\
 &\phantom{\le 4}\times \mathbb{E}\left( \chi_{A_k\cap A_{n,\textbf{i}}}(\omega_k)
   \chi_{A_k\cap A_{n,\textbf{j}}}(\omega_k) b_k^{-1}\right) \\
 &=4\Vert \rho\Vert_{\infty}^{2}\Bigl(1-\frac{1}{q_n}\Bigr)^{-2}\mu(F_n)^{-2}
   \sum_{k=M_n}^{N_n}b_{k}\,\sum_{\textbf{i}\in \Sigma_{n}'}\mu(A_{n,\textbf{i}})^{2}
   a_{n,\textbf{i}}^{2}\,\mu\left( A_k\cap A_{n,\textbf{i}}\right)\\
 &=4\Vert \rho\Vert_{\infty}^{2}\Bigl(1-\frac{1}{q_n}\Bigr)^{-2}\mu(F_n)^{-2}
   \sum_{\textbf{i}\in \Sigma_{n}'}\mu(A_{n,\textbf{i}})c_{n,\textbf{i}}^{-1}
   a_{n,\textbf{i}}\sum_{k=M_n}^{N_n}b_{k}\,\mu\left( A_k\cap A_{n,\textbf{i}}\right)\\
   &\leq 4\Vert \rho\Vert_{\infty}^{2}\Bigl(1-\frac{1}{q_n}\Bigr)^{-2}\mu(F_n)^{-1}
   2^{-n}\xrightarrow[n\to\infty]{} 0.
\end{align*}
To estimate $S_2$, let $\delta>0$ and let $n\in \N$ be large enough so that 
\begin{equation}\label{equation:boundsforcontinuousrho}
    \rho(x)\rho(y)\leq \rho(u)\rho(v)+\delta
\end{equation}
whenever $x,y,u,v\in X$ are such that $\max\{ |x-u|,|y-v| \}\leq r_{M_n}$.
Let $\beta:=2^{-n}$ so that $r_{M_n}<\frac{\beta^2}{2}$ (recall \eqref{Fn}).
We now write
\begin{align*}
\mathbb{E}&\left( \chi_{A_l\cap A_{n,\textbf{i}}\cap H_{q_n}}(\omega_l)
   \chi_{A_m\cap A_{n,\textbf{j}}\cap H_{q_n}}(\omega_m)J_t(\psi_{l}^{\omega}\rho\,d\mu_{|_G},
   \psi_{m}^{\omega}\rho\, d\mu_{|_G})\right)\\
 &= E_1(l,m,\textbf{i},\textbf{j})+E_2(l,m,\textbf{i},\textbf{j}),
\end{align*}
where 
\begin{align}
E_1(l,m,\textbf{i},\textbf{j})&:=\int\int_{|\omega_{l}-\omega_{m}|\geq \beta}
   \chi_{A_l\cap A_{n,\textbf{i}}\cap H_{q_n}}(\omega_l)
   \chi_{A_m\cap A_{n,\textbf{j}}\cap H_{q_n}}(\omega_m)\\
 &\phantom{tahanvahantilaa}\times J_t(\psi_{l}^{\omega}\rho\, d\mu_{|_{G}},
   \psi_{m}^{\omega}\rho\,d\mu_{|_{G}})\,d\mu(\omega_l)\,d\mu(\omega_m)\,\text{ and}
   \nonumber\\
E_2(l,m,\textbf{i},\textbf{j})&:=\int\int_{|\omega_{l}-\omega_{m}|< \beta}
   \chi_{A_l\cap A_{n,\textbf{i}}\cap H_{q_n}}(\omega_l)
   \chi_{A_m\cap A_{n,\textbf{j}}\cap H_{q_n}}(\omega_m)\\
 &\phantom{tahanvahantilaa}\times J_t(\psi_{l}^{\omega}\rho\, d\mu_{|_{G}},
   \psi_{m}^{\omega}\rho\,d\mu_{|_{G}})\,d\mu(\omega_l)\,d\mu(\omega_m).\nonumber
\end{align}

We begin by estimating $E_1(l,m,\textbf{i},\textbf{j})$ for $l\neq m$ (or more
precisely, the contribution of the $E_1$-terms to $S_2$). Observe that if 
$|\omega_{l}-\omega_{m}|\geq \beta$, then 
\[|x-y|\geq (1-\beta)|\omega_{l}-\omega_{m}|
\]
for $x\in B_{l}^{\omega}$ and $y\in B_{m}^{\omega}$. 
Thus, recalling \eqref{equation:defnofpsi_k^omega},
\eqref{equation:boundsforcontinuousrho}, the facts that $H_{q_n}\subset G$ and
$\psi_k^\omega\,d\mu_{|_G}$ is a probability measure and \eqref{munismuFn}, we
obtain  
\begin{align}\label{equation:energycontributionfromE_1terms}
\mu&(F_n)^{-2}\sum_{\substack{l,m=M_n\\ l\neq m}}^{N_n}b_{l}\,b_{m}\,
   \sum_{\textbf{i},\textbf{j}\in \Sigma_{n}'}\mu(A_{n,\textbf{i}})\mu(A_{n,\textbf{j}})
   a_{n,\textbf{i}}a_{n,\textbf{j}} E_1(l,m,\textbf{i},\textbf{j})\nonumber\\
 &\leq (1-\beta)^{-t}\mu(F_n)^{-2}\sum_{\substack{l,m=M_n\\ l\neq m}}^{N_n}b_{l}\,b_{m}\,
   \sum_{\textbf{i},\textbf{j}\in \Sigma_{n}'}\mu(A_{n,\textbf{i}})\mu(A_{n,\textbf{j}})
   a_{n,\textbf{i}}a_{n,\textbf{j}}\nonumber\\
 &\phantom{\le}\times\int\int_{|\omega_{l}-\omega_{m}|\geq \beta}
   \chi_{A_l\cap A_{n,\textbf{i}}\cap H_{q_n}}(\omega_l)\chi_{A_m\cap A_{n,\textbf{j}}\cap H_{q_n}}
   (\omega_m)|\omega_l-\omega_m|^{-t}\nonumber\\
 &\phantom{\le}\times(\rho(\omega_l)\rho(\omega_m)+\delta)\int\int
   \psi_{l}^{\omega}(x)\psi_{m}^{\omega}(y) \,d\mu_{|_G}(x)\,d\mu_{|_G}(y)\,
   d\mu(\omega_l)\,d\mu(\omega_m)\nonumber\\
 &\leq (1-\beta)^{-t}\mu(F_n)^{-2}\sum_{l,m=M_n}^{N_n}b_{l}\,b_{m}\,
   \sum_{\textbf{i},\textbf{j}\in \Sigma_{n}'}\mu(A_{n,\textbf{i}})\mu(A_{n,\textbf{j}})
   a_{n,\textbf{i}}a_{n,\textbf{j}}\nonumber\\
 &\phantom{\le}\times\int_{G}\int_{G}\chi_{A_l\cap A_{n,\textbf{i}}}(\omega_l)
   \chi_{A_m\cap A_{n,\textbf{j}}}(\omega_m)  
   |\omega_l-\omega_m|^{-t} (\rho(\omega_l)\rho(\omega_m)+\delta)
   \,d\mu(\omega_l)\,d\mu(\omega_m)\nonumber\\
 &=(1-\beta)^{-t}\mu(F_n)^{-2}\int_{G}\int_{G} |z-v|^{-t}(\rho(z)\rho(v)+\delta)\,
   d\mu_{|_{F_n}}(z)\, d\mu_{|_{F_n}}(v)\nonumber\\
 &\leq (1-\beta)^{-t}\mu(F_n)^{-2}(I_t(\rho \,d\,\mu_{|_G})+\delta I_t(\mu_{|_G}) ).
\end{align}

Next we estimate the $E_2$-terms. To this end, for $x\in G$, let
\[
\widetilde{E}_{n}(x)\coloneqq \mu(F_n)^{-1}\sum_{l=M_n}^{N_n}b_{l}\,
\sum_{\textbf{i}\in \Sigma_{n}'}\mu(A_{n,\textbf{i}})a_{n,\textbf{i}}
\int_{A_l\cap A_{n,\textbf{i}}\cap H_{q_n}}\psi_{l}^{\omega}(x)\, d\mu(\omega_l). 
\]
Note that $\widetilde{E}_{n}(x)$ is a Borel function. Recalling
\eqref{equation:defnofpsi_k^omega}, \eqref{Hn}, \eqref{En}, \eqref{G} and the
fact $\chi_{B(\omega_l,r_l)}(x)=\chi_{B(x,r_l)(\omega_l)}$, we estimate, for $x\in G$
and $n\geq N$,
\begin{align}\label{equation:estimatefore_ntilde}
\widetilde{E}_{n}(x)&\leq 2\Bigl(1-\frac{1}{q_n}\Bigr)^{-1}\mu(F_n)^{-1}
   \sum_{l=M_n}^{N_n}b_{l}\,\sum_{\textbf{i}\in \Sigma_{n}'}\mu(A_{n,\textbf{i}})
   a_{n,\textbf{i}}\int_{A_l\cap A_{n,\textbf{i}}\cap H_{q_n}} \frac{\chi_{B_l^{\omega}}(x)}
   {\mu(B_l^{\omega})}\, d\mu(\omega_l)\nonumber\\
 &\leq 2\Bigl(1-\frac{1}{q_n}\Bigr)^{-1}E_n(x)\leq 2\Bigl(1-\frac{1}{q_n}
   \Bigr)^{-1}n^2.
\end{align}
Observe that if $|\omega_l-\omega_m|<\beta$, $x\in B_l^{\omega}$ and
$y\in B_m^{\omega}$, then $|x-y|<2\beta$, since $r_{M_n}<\beta^2/2$ and the
sequence $(r_k)_{k=1}^\infty$ is decreasing. Thus, by
\eqref{equation:estimatefore_ntilde}) and
Lemma~\ref{lemma:energyupperboundfrostmanmeasures},
\begin{align}\label{equation:energycontributionfromE_2terms}
\mu&(F_n)^{-2}\sum_{\substack{l,m=M_n\\ l\neq m}}^{N_n}b_{l}\,b_{m}\,
    \sum_{\textbf{i},\textbf{j}\in \Sigma_{n}'}\mu(A_{n,\textbf{i}})\mu(A_{n,\textbf{j}})
    a_{n,\textbf{i}}a_{n,\textbf{j}} E_2(l,m,\textbf{i},\textbf{j})\nonumber\\
 &\leq \Vert\rho\Vert_{\infty}^{2}\mu(F_n)^{-2}\sum_{\substack{l,m=M_n\\ l\neq m}}^{N_n}
    b_{l}\,b_{m}\,\sum_{\textbf{i},\textbf{j}\in \Sigma_{n}'}\mu(A_{n,\textbf{i}})
    \mu(A_{n,\textbf{j}})a_{n,\textbf{i}}a_{n,\textbf{j}}\nonumber\\
 &\phantom{\le 2}\times\int\int_{|\omega_{l}-\omega_{m}|< \beta}
    \chi_{A_l\cap A_{n,\textbf{i}}\cap H_{q_n}}(\omega_l)\chi_{A_m\cap A_{n,\textbf{j}}\cap H_{q_n}}
    (\omega_m)  \nonumber\\
 &\phantom{\le 2}\times\int \int_{|x-y|<2\beta} |x-y|^{-t} \psi_{l}^{\omega}(x)
    \psi_m^{\omega}(y)\,d\mu_{|_{G}}(x)\,d\mu_{|_{G}}(y)\,d\mu(\omega_l)\,
    d\mu(\omega_m)\nonumber\\
 &\leq \Vert\rho\Vert_{\infty}^{2}\int\int_{|x-y|<2\beta} |x-y|^{-t}
    \widetilde{E}_n(x)\widetilde{E}_n(y)\,d\mu_{|_{G}}(x)\,
    d\mu_{|_{G}}(y)\nonumber\\
 &\leq 4 \Vert \rho\Vert_{\infty}^{2}\Bigl(1-\frac{1}{q_n}\Bigr)^{-2}
    n^4\int \int_{|x-y|<2\beta} |x-y|^{-t}\,d\mu(x)\,d\mu(y)\nonumber\\
 &\leq C_{s,t} \Vert \rho\Vert_{\infty}^{2} n^4\beta^{s-t}.
  \end{align}
Recalling that $\beta=2^{-n}$, we obtain by combining the estimates from
\eqref{equation:energycontributionfromE_1terms},
\eqref{equation:energycontributionfromE_2terms} and \eqref{Fn} that 
\[
S_2\leq (1-2^{-n})^{-t}\mu(F_n)^{-2}(I_t(\rho \,d\,\mu_{|_G})+\delta I_t(\mu_{|_G}))
+C_{s,t}\Vert \rho\Vert_{\infty}^{2}n^4\,2^{-n(s-t)}.
\]
Letting $\delta\to 0$  (and thus also $n\to \infty$) yields
\begin{equation}\label{equation:claimforcontinuousfunctions}
\limsup_{n\to \infty}  \mathbb{E}\left( I_t(\varphi_{n}^{\omega}\rho\,d\mu_{|_{G}})
\right)\leq I_t(\rho\,d\mu_{|_G}).
\end{equation}
Thus, by Lemma~\ref{lemma:liminfatmostexpectation},
\[
\liminf_{n\to\infty}I_t(\varphi^{\omega}_{n}\rho\,d\mu_{|_{G}})\le I_t(\rho\,d\mu_{|_G})
\]
almost surely. Since $\rho\in C_0(\R^d)$ was arbitrary, we obtain
that 
\[
\mathbb{P}\left(\Omega_0\right)\coloneqq\mathbb{P}\bigl(\{\liminf_{n\to \infty}
I_t(\varphi_{n}^{\omega}\rho_i\,d\mu_{|_{G}})\leq I_t(\rho_{i}\,d\mu_{|_G})
\text{ for all } i\in \N\}\bigr)=1,
\]
where $\{ \rho_i\}_{i\in \N}$ is a dense subset of the separable space
$C_0(\R^d)$ with the property that $\rho_j+q\in\{\rho_i\}_{i\in\N}$
for every $j\in\N$ and for every non-negative $q\in\Q$. Let
$\omega\in \Omega_0$, $\varepsilon_{0}>0$ and let $\rho$ be any finite product
of the functions $\{\varphi_n^{\omega}\}_{n\in\N}$. Then there is
$i\in\N$ such that $\rho_i\geq \rho$ and
$\Vert \rho-\rho_i \Vert_{\infty}<\varepsilon_{0}$. Thus,
\begin{align*}
\liminf_{n\to\infty}I_t(\varphi_n^{\omega}\rho\,d\mu_{|_{G}})&\leq \liminf_{n\to\infty}
    I_t(\varphi_n^{\omega}\rho_i\,d\mu_{|_{G}})\leq I_t(\rho_i\,d\mu_{|_G})\\
 &\leq I_t(\rho\,d\mu_{|_G})+ (2\varepsilon_{0}\Vert \rho\Vert_{\infty}
    +\varepsilon_{0}^2)I_t(\mu_{|_G}).
\end{align*}
Since $\varepsilon_{0}>0$ was arbitrary and $I_t(\mu_{|_G})<\infty$, we obtain
that 
\[
\liminf_{n\to\infty}I_t(\varphi_n^{\omega}\rho\,d\mu_{|_{G}})\leq I_t(\rho\,d\mu_{|_G}).
\]
Since $\omega\in \Omega_0$ was arbitrary, we have that the assumptions of
Lemma~\ref{lemma:deterministiclowerbound} are satisfied almost surely.
By \eqref{varphinomega} and \eqref{equation:defnofpsi_k^omega}, we have
that $\spt\varphi_n^\omega\subset\bigcup_{k=M_n}^{N_n}B(\omega_k,r_k)$, implying
that $E_{\underline r}(\omega)\supset\limsup_{n\to\infty}(\spt\varphi_n^\omega)$.
Hence, by Lemma~\ref{lemma:deterministiclowerbound},
\[
\capa_{t}\left(E_{\underline{r}}(\omega)\right)\geq\frac{\mu(G)^{2}}{I_t(\mu_{|_G})}>0
\]
almost surely, which implies that
$\dim_{\mathrm{H}}\left(E_{\underline{r}}(\omega) \right)\geq t$ almost surely.
\end{proof}

The following corollary of Theorem~\ref{theorem:generallowerboundenergy} will
be convenient for us in the proofs of Theorem~\ref{theorem:nminusalpha} and
Theorem~\ref{theorem:lowerboundforgeneralballs}.

\begin{corollary}\label{corollary:generallowerboundmeasure}
Let $\mu\in\mathcal{P}(\R^d)$ be $(C,s)$-Frostman. Let
$(r_k)_{k=1}^{\infty}$ be a decreasing sequence of positive numbers tending to
zero. For $u>0$, let 
\begin{equation}\label{Aku}
\widehat A_k^{u}:=\left\{x\in\spt\mu\mid\mu(B(x,r_k))\geq r_k^{u}\right\}.
\end{equation}
If for some $u>0$ and $0<t<s$ it is true that
\begin{equation}\label{equation:measureassumptionforlowerbound}
\mu\Bigl(\bigl\{ x\in\spt\mu\mid\sum_{k=1}^\infty\chi_{\widehat A_k^{u}}(x)
  r_k^{\frac{tu}{s}}=\infty\bigr\}\Bigr)=1,
\end{equation}
then $f_{\mu}(\underline{r})\geq t$.
\end{corollary}

\begin{proof}
For $x\in\widehat A_k^u$,
Lemma~\ref{lemma:energyupperboundfrostmanmeasures} yields
\[
I_{t}(\mu_{B(x,r_k)})\leq\frac{C^{\frac ts}s}{s-t}\mu(B(x,r_k))^{-\frac ts}
\leq \frac{C^{\frac ts}s}{s-t}r_k^{-\frac{tu}{s}},
\]
hence the assumptions of Theorem~\ref{theorem:generallowerboundenergy} are
satisfied with
\[
b_k=\biggl(\frac{C^{\frac ts}s}{s-t}\biggr)^{-1}r_k^{\frac{tu}{s}}.
\]
Thus $f_{\mu}(\underline{r})\geq t$.
  \end{proof}

\section{Proof of Theorem \ref{theorem:nminusalpha}}\label{proofofEP}

To prove Theorem \ref{theorem:nminusalpha}, we need the following lemma.

\begin{lemma}\label{lemma:pointwisediv}
Let $\mu\in \mc{P}(\R^d)$, $\alpha>0$ and let
$s\geq \overline{\dim}_{\mathrm{H}}\,\mu$. Write $r_k:=k^{-\alpha}$ for
all $k\in\N$. Then for every $\varepsilon>0$ and
$t<\frac{1}{\alpha}$, we have that
\[
\mu\Bigl(\bigl\{x\in\spt\mu\mid\sum_{k=1}^{\infty}r_k^{t}
\chi_{\widehat A_k^{s+\varepsilon}}(x)=\infty\bigr\}\Bigr)=1.
\]
\end{lemma}

\begin{proof}
Since $s\geq \overline{\dim}_{\mathrm{H}}\,\mu$, we have that
$\dimlocl\mu(x)\leq s$ for $\mu$-almost every $x\in X$. Fix such a point $x$.
For every $\varepsilon>0$, define the set
\[
A^{s+\varepsilon}(x):=\{k\in \N\mid \mu(B(x,r_k))\geq r_k^{s+\varepsilon}\}.
\]
Since $r_k=k^{-\alpha}$, it is true that
\[
\dimlocl \mu(x)=\liminf_{k\to \infty}\frac{\log \mu(B(x,r_k))}{\log r_k}.
\]
Thus the number of elements in $A^{s+\varepsilon}(x)$ is infinite for
every $\varepsilon>0$. Observe now that if $k\in A^{s+\varepsilon}(x)$
and $l<k$ is such that $l\not\in A^{s+2\varepsilon}(x)$, then 
\[
k^{-\alpha(s+\varepsilon)}=r_k^{s+\varepsilon}\leq \mu(B(x,r_k))\leq \mu(B(x,r_l))
<r_{l}^{s+2\varepsilon}=l^{-\alpha(s+2\varepsilon)}.
\]
Hence $k>l^{\gamma}$, where
$\gamma\coloneqq 1+\frac{\varepsilon}{s+\varepsilon}>1$. Thus we find
arbitrarily large natural numbers $k$ such that
$\{\lceil k^{1/\gamma}\rceil,\lceil k^{1/\gamma}\rceil+1,\ldots,k\}\subset
A^{s+2\varepsilon}(x)$.
Thus, for any $t<\frac{1}{\alpha}$, 
\begin{align*}
\sum_{k=1}^{\infty}\chi_{\widehat A_k^{s+2\varepsilon}}(x)r_{k}^{t}
  &=\sum_{k\in A^{s+2\varepsilon}(x)}r_k^{t}
    \geq\liminf_{k\to\infty}\sum_{j=\lceil k^{1/\gamma}\rceil}^{k}r_{j}^{t}
    \geq\liminf_{k\to\infty}\sum_{j=\frac k2}^{k}r_{j}^{t}\\
  &\geq\liminf_{k\to\infty}\frac{k}{2} k^{-\alpha t}=\infty,
\end{align*}
since $1-\alpha t>0$. Since $\varepsilon>0$ was arbitrary, the claim follows. 
\end{proof}

\begin{proof}[Proof of Theorem \ref{theorem:nminusalpha}]
Assume first that $\frac 1\alpha<\overline{\dim}_{\mathrm{H}}\,\mu$.
Fix $0<t<\frac{1}{\alpha}$. Our aim is to show that $f_{\mu}(\alpha)\geq t$.
Write $s:=\overline{\dim}_{\mathrm{H}}\,\mu$ and choose $\varepsilon>0$ small
enough so that $\frac{1}{\alpha}<s-\varepsilon$ and
\begin{equation}\label{equation:smallenoughepsilon}
t\,\frac{s+\varepsilon}{s-\varepsilon}<\frac{1}{\alpha}.
\end{equation}
By definition of $s$, we can choose $C>0$ such that $\mu(D)>0$, where 
\[
D:=\{x\in\spt\mu\mid \mu(B(x,r))\leq Cr^{s-\varepsilon}
\text{ for all } r>0\}.
\]
Applying Lemma~\ref{lemma:ekstromperssonboundeddensities} with
$\mu=\mu_D$ and $\nu=\mu$, we conclude that it suffices to show that 
$f_{\mu_D}(\alpha)\geq t$. Observe that there exists $C'>0$ such that
\begin{equation}\label{muD}
\mu_D(B(x,r))\leq C'r^{s-\varepsilon}
\end{equation}
for every $x\in\R^d$ and every $r>0$. By Lebesgue's density theorem, we have
that 
\[
\underline{\dim}_{\mathrm{loc}}\mu_{D}(x)=\underline{\dim}_{\mathrm{loc}}\mu(x)
\]
for $\mu_{D}$-almost every $x\in\R^d$ and, thus,
$\overline{\dim}_{\mathrm{H}}\,\mu_{D}\leq s$. 
Using \eqref{equation:smallenoughepsilon} and
Lemma~\ref{lemma:pointwisediv} for $\mu_D$, we conclude that
\[
\sum_{k=1}^\infty r_k^{ t\,\frac{s+\varepsilon}{s-\varepsilon}}\chi_{\widehat A_k^{s+\vep}}(x)
=\infty
\]  
for $\mu_{D}$-almost every $x\in\R^d$. Applying
Corollary~\ref{corollary:generallowerboundmeasure} for $\mu_D$ with $u=s+\vep$
and $s$ replaced by $s-\vep$ (see \eqref{muD}) yields
\[
f_{\mu_{D}}(\alpha)\geq t.
\]

The claim is true also in the case
$\frac 1\alpha=\overline{\dim}_{\mathrm{H}}\,\mu$ since
$E_\beta(\omega)\subset E_\alpha(\omega)$ for $\frac 1\beta<\frac 1\alpha$ and
$f_\mu(\alpha)\le\frac 1\alpha$. 
\end{proof}

\section{Proof of Theorem \ref{theorem:lowerboundforgeneralballs}}
 \label{proofofgeneral}

\begin{proof}[Proof of Theorem~\ref{theorem:lowerboundforgeneralballs}]
Let $0<\varepsilon<\overline\delta$. By the definition of
$\overline{\delta}$, there exists a Borel set $H_0\subset\spt\mu$
such that $\mu(H_0)>0$ and, for every $x\in H_0$, 
\[
\frac{\dimlocl \mu(x)}{\dimlocu \mu(x)}> \overline{\delta}-\varepsilon
\text{ and } \dimlocl \mu(x)> s_2(\underline r).
\]
For every small enough $\gamma>0$ there exists a Borel set $H_1\subseteq H_0$
with $\mu(H_1)>0$ such that
$\dimlocl \mu(x)>s_2(\underline r)+\gamma$ for every $x\in H_1$.
Further, for some $s\geq s_2(\underline r)$, we find a Borel set
$H_2\subseteq H_1$ with $\mu(H_2)>0$ such that 
$\dimlocl \mu(x)\in [s+\gamma,s+2\gamma]$. Then, for every $x\in H_2$ we have
that
\[
s+\gamma\leq \dimlocl \mu(x)\leq \dimlocu \mu(x)
\leq \frac{\dimlocl \mu(x)}{\overline{\delta}-\varepsilon}
\leq \frac{s+2\gamma}{\overline{\delta}-\varepsilon}.
\]
Finally, there exists a Borel set $H_3\subseteq H_2$ with positive $\mu$-measure
such that
\[
r^{\frac{s+3\gamma}{\overline{\delta}-\varepsilon}}\leq \mu(B(x,r))\leq r^{s+\frac{\gamma}{2}}
\]
for every $x\in H_3$ and for all sufficiently small $r>0$. Thus
for some constant $C=C(\mu,\gamma,H_3)$, we have that 
\[
\mu_{H_3}(B(x,r))\leq Cr^{s+\frac{\gamma}2}
\]
for every $x\in\R^d$ and $r>0$.
    
We will now apply Corollary~\ref{corollary:generallowerboundmeasure} to
$\mu_{H_3}$ with $u=\frac{s+4\gamma}{\overline\delta-\vep}$ and
$s$ replaced by $s+\frac{\gamma}2$. By Lebesgue's density theorem,
\[
\lim_{r\to 0}\frac{\mu(H_3\cap B(x,r))}{\mu(B(x,r))}=1
\]
for $\mu$-almost every $x\in H_3$, hence $\mu_{H_3}$-almost every $x$ belongs to
$\widehat A_k^u$ (see \eqref{Aku}) for every $k$ large enough
(depending on $x$). Now observe that if
$t<s_2(\underline r)(\overline{\delta}-\varepsilon)
\frac{s+\frac{\gamma}{2}}{s+4\gamma}$, then
\[
t\frac{s+4\gamma}{(\overline{\delta}-\varepsilon)(s+\frac{\gamma}{2})}
<s_2(\underline r).
\]
Recalling \eqref{s2}, we have that
\[
\sum_{k=1}^\infty\chi_{\widehat A_k^u}(x)
r_k^{t \frac{s+4\gamma}{\left(\overline{\delta}-\varepsilon\right)(s+\frac{\gamma}{2})}}=\infty
\]
for $\mu_{H_3}$-almost every $x\in\R^d$. By
Corollary~\ref{corollary:generallowerboundmeasure},
$f_{\mu_{H_3}}(\underline{r})\geq t$.
Lemma~\ref{lemma:ekstromperssonboundeddensities} now yields
$f_{\mu}(\underline{r})\geq t$. From this we deduce that
$f_{\mu}(\underline{r})\geq s_2(\underline r)(\overline{\delta}-\varepsilon)
\frac{s+\frac{\gamma}{2}}{s+4\gamma}$ and letting $\gamma \to 0$ and
$\varepsilon \to 0$ completes the proof.
\end{proof}

\section{Sharpness of the bounds in Theorem
  \ref{theorem:lowerboundforgeneralballs}}\label{examples}

In this section, we provide examples which demonstrate that the bounds in
Theorem~\ref{theorem:lowerboundforgeneralballs} are sharp. In
particular, these examples show that there is no formula for
$f_{\mu}(\underline{r})$ involving only the quantity $s_2(\underline r)$ and the
local dimensions of the measure. For simplicity, we do
our constructions in $\R$ but similar examples can be constructed also in
$\R^d$.

We start with an example showing that it is possible that
$f_{\mu}(\underline{r})=s_2(\underline{r})$ for every sequence $\underline{r}$
with $s_2(\underline{r})<\dimhu\mu$, no matter how small the quantity
$\overline{\delta}$ is. In particular, this example shows that for some
measures, the trivial upper bound in
Theorem~\ref{theorem:lowerboundforgeneralballs} is the correct value for the
almost sure dimension for all sequences of radii.

\begin{example}\label{example:dimensionalwayss_2}
Let $0<s<u<1$ and let $0<\alpha,\beta<\frac{1}{2}$ be such that
$s=\frac{\log 2}{-\log \alpha}$ and $u=\frac{\log 2}{-\log \beta}$. Let
$(N_k)_{k=1}^\infty$ be a rapidly growing sequence of integers and let
$C:=\bigcap_{n=0}^{\infty}C_n$ be the Cantor set, where $C_{n}$ is obtained from
$C_{n-1}$ by removing the middle $(1-2\alpha)$-interval from each
construction interval of $C_{n-1}$ for $N_{2k}\leq n< N_{2k+1}$ and by removing
the middle $(1-2\beta)$-interval from each construction interval of
$C_{n-1}$ for $N_{2k+1}\leq n<N_{2k+2}$. Let $\mu$ be the natural 
Borel probability measure on $C$ giving equal weight to all construction
intervals at same level. We then have that 
\begin{equation}\label{equation:estimatesforcantormeasure}
r^{u}\lesssim \mu(B(x,r)) \lesssim r^{s}    
\end{equation}
for all $x\in C$ and $0<r<1$, where the notation $g(r)\lesssim h(r)$
 means that there exists a constant $C$ such that $g(r)\le Ch(r)$. If the
sequence $(N_k)_{k=1}^\infty$ grows fast enough, then $\dimlocl \mu(x)=s$ and
$\dimlocu \mu(x)=u$ for all $x\in C$.

Our aim is to show that $f_{\mu}(\underline{r})=s_2(\underline{r})$
for any sequence $\underline{r}$ with $s_2(\underline{r})<s$. To this end, let
$x_0\in C$ and $0<r_0<1$. Let $I_0$ be the largest construction
interval contained in $C\cap B(x_0,r_0)$. Let $n_0$ be the level of $I_0$. Then
$\mu(I_0)\approx\mu(B(x_0,r_0))$, $|I_0|\approx r_0$ and
\begin{equation}\label{ItI0}
I_t(\mu_{B(x_0,r_0)})\lesssim I_t(\mu_{I_0})
\end{equation}
for all $t>0$. Here $g\approx h$ means that $g\lesssim h$ and $h\lesssim g$.
Let $0<t<s$ and fix $x\in I_0$. Let $N>n_0$ be large and for
$j\in \{ n_0,\ldots ,N-1\}$, set 
\[
D_j(x):=\{y\in C\mid C_j(y)=C_j(x)\text{ and }C_{j+1}(y)\cap C_{j+1}(x)
=\emptyset\},
\]
where $C_j(y)$ denotes the unique construction interval of $C_j$ containing $y$.
Note that $D_j(x)$ is a construction interval of $C_{j+1}$ and 
$I_0=C_N(x)\cup\bigcup_{j=n_0}^{N-1}D_j(x)$ with the union disjoint.
Let $\gamma_{j}$ denote the length of the intervals removed from $C_{j-1}$ in the
construction to obtain $C_j$. Since $\alpha<\beta$, we have that 
\begin{equation}\label{equation:lengthofgapincantorset}
\dist(x,D_j(x))\geq \gamma_{j+1}\geq (1-2\beta)\alpha^{j-n_0}|I_0|
\end{equation}
for $j\in \{ n_0,\ldots,N-1\}$. Finally, let $\ell_n$ denote the length of the
construction intervals of $C_n$. By
Lemma~\ref{lemma:energyupperboundfrostmanmeasures},
\eqref{equation:lengthofgapincantorset} and
\eqref{equation:estimatesforcantormeasure}, we obtain
\begin{align*}
\phi_{\mu_{|_{I_0}}}^{t}(x)&=\int_{I_0}|x-y|^{-t}\,d\mu(y)=\int_{C_N(x)}|x-y|^{-t}\,
    d\mu(y)+\sum_{j=n_0}^{N-1}\int_{D_j(x)}|x-y|^{-t}\,d\mu(y)\\
 &\lesssim |C_N(x)|^{s-t}+\sum_{j=n_0}^{N-1}\mu(D_j(x))\gamma_{j+1}^{-t}
    = \ell_N^{s-t}+\sum_{j=n_0}^{N-1}2^{n_0-j-1}\mu(I_0)\gamma_{j+1}^{-t}\\
 &\leq \ell_N^{s-t}+\mu(I_0)(1-2\beta)^{-t}|I_0|^{-t}\sum_{j=n_0}^{N-1}2^{n_0-j-1}
    \alpha^{t(n_0-j)}\\
 &\le \ell_N^{s-t}+\frac 12\mu(I_0)(1-2\beta)^{-t}|I_0|^{-t}\sum_{j=0}^{\infty}\Bigl(
    \frac{\alpha^{-t}}{2}\Bigr)^{j}.
\end{align*}
Since $t<s=\frac{\log 2}{-\log \alpha},$ it is true that $\alpha^{-t}<2$. Hence,
the series above converges. Integrating over $I_0$ and normalising, we obtain
for some $C_1=C_{\alpha,\beta,t}$ that
\[
I_t(\mu_{I_0})\le C_1\bigl(\mu(I_0)^{-1}\ell_N^{s-t}+|I_0|^{-t}\bigr).
\]
Letting $N\to \infty$ yields $I_t(\mu_{I_0})\le C_1|I_0|^{-t}$.
Recalling \eqref{ItI0}, we obtain the estimate 
$I_t(\mu_{B(x_0,r_0)})\lesssim |I_0|^{-t} \approx r_0^{-t}$.
If $\underline{r}$ is a sequence of positive numbers such that
$s_2(\underline{r})<s$ and $0<t<s_2(\underline{r})$, then there
exists a constant $C_2$ such that $I_t(\mu_{B(x,r_k)})\le C_2r_k^{-t}$ for all
$x\in C$ and $k\in \N$, hence Theorem~\ref{theorem:generallowerboundenergy}
implies that $f_{\mu}(\underline{r})\geq t$. Letting $t\to s_2(\underline{r})$
through a countable sequence then yields the claim
$f_{\mu}(\underline{r})=s_2(\underline{r})$.
\end{example}

In the next example, we construct a measure $\mu$ such that the
quantity $\overline{\delta}$ can be made arbitrarily small, and for any
$\gamma\in [\overline{\delta},1]$, there exists a sequence $\underline{r}$ of
radii such that $f_{\mu}(\underline{r})=s_2(\underline{r})\gamma$.
In particular, this example shows that the lower bound in
Theorem~\ref{theorem:lowerboundforgeneralballs} can be attained.

\begin{example}\label{example:lowerboundmaybesharp}
Let $0<s<u<1$. We will first construct a measure $\mu \in \mc{P}([0,1])$ such
that $\dimlocl \mu(x)=s$ and $\dimlocu\mu(x)=u$ for every $x\in\spt\mu$.
Fix some $\ell_{0}\in ]0,1[$ and define a sequence
$(\ell_k)_{k=0}^\infty$ of positive numbers by setting $\ell_{k+1}:=\ell_{k}^{v}$,
where 
\[
v\coloneqq \frac{u(1-s)}{s(1-u)}.
\]
For every $k\in\N\setminus\{0\}$, let
$L_k\coloneqq\ell_{k}^{\frac{s}{u}}=\ell_{k-1}^{\frac{1-s}{1-u}}$. Note that
$\ell_{k}<L_k<\ell_{k-1}$ for every $k\in\N\setminus\{0\}$ and
$(l_k)_{k=0}^\infty$ tends to zero with super exponential speed. Let
$C_0=[0,\ell_0]$. Construct $C_1$ by partitioning $C_0$ into 
\[
N_1\coloneqq\left\lfloor \frac{\ell_{0}}{L_1}\right\rfloor
\]
intervals of length $L_1$, and from each of these intervals, keep the leftmost
segment of length $\ell_{1}$ and discard the rest to obtain $C_1$. If $C_k$ has
been constructed and $C_k$ is a disjoint union of $\prod_{j=1}^{k}N_j$ intervals
of length $\ell_{k}$, construct $C_{k+1}$ by dividing each interval of $C_k$ into
$N_{k+1}\coloneqq \lfloor\frac{\ell_{k}}{L_{k+1}}\rfloor$ intervals of length
$L_{k+1}$ and from each of these intervals  keeping the leftmost segment of
length $\ell_{k+1}$. This way we obtain a decreasing sequence $(C_k)_{k=0}^\infty$
of compact sets with the properties that each $C_k$ is a union of
$\prod_{j=1}^{k}N_j$ intervals of length $\ell_{k}$ and these intervals are
at least $(L_k-\ell_k)$-separated. Note that
$L_k-\ell_k>\frac{L_k}{2}$ for all large $k\in\N$. Set
$C:=\bigcap_{k=0}^{\infty}C_k$. Let $\mu$ be the natural uniformly distributed
Borel probability measure on $C$. Since $(l_k)_{k=0}^\infty$ tends to zero fast,
$(N_k)_{k=1}^\infty$ tends to infinity fast. Therefore, the facts  
$N_k=\lfloor\frac{\ell_{k-1}}{L_{k}}\rfloor$ and
\[
\frac{\ell_{k-1}}{L_k}\ell_{k}^{s}=\ell_{k-1}^{1+vs(1-\frac{1}{u})}=\ell_{k-1}^{s}
\]
imply that $\mu(B(x,r))\lesssim r^{s}$ for all $x\in\R$ and $r>0$. Furthermore,
for all $x\in C$, we have that 
\begin{equation}\label{extremescales}
\mu(B(x,\tfrac{L_k}2))=\mu(B(x,\ell_{k}))\approx \ell_{k}^{s}=L_k^{u}
\end{equation}  
for all $k\in\N$, and $\mu(B(x,r))\gtrsim r^{u}$ for
all $0<r<1$. Thus, $\dimlocl \mu(x)=s$ and $\dimlocu \mu(x)=u$ for
all $x\in C$ and, in particular, $\overline{\delta}=\frac{s}{u}$,
where $\overline{\delta}$ is the quantity defined in
Theorem~\ref{theorem:lowerboundforgeneralballs}.

In the following, we will demonstrate that by choosing a suitable sequence $\underline{r}$, the almost sure dimension $f_{\mu}(\underline{r})$ can take any value in the interval $[s_2(\underline{r})\frac{s}{u},s_2(\underline{r})]$, that is, the possible pairs $(s_2(\underline{r}),f_{\mu}(\underline{r}))$ fill the entire triangle depicted in Figure \ref{figure:possiblevaluesfordimh}.
\begin{figure}[h!]
\begin{center}
        \begin{tikzpicture}[line cap=round, line join=round, >=triangle 45,
                            x=1.0cm, y=1.0cm, scale=1]

          \draw [->,color=black] (-0.1,0) -- (10,0);
          \draw [->,color=black] (0,-0.1) -- (0,4.5);
          \draw (0,0) -- (3,3);
          \draw (0,0) -- (8,3);
          \draw (3,3) -- (8,3);
          
          \draw (0,4.3) node [left] {$f_{\mu}(\underline{r})$};
          \draw (10,0) node [below] {$s_2(\underline{r})$};

          \draw[shift={(3,0)}] (0pt, 2pt) -- (0pt,-2pt) node [below] {$\dim_{\mathrm{H}}\mu$};
          \draw[shift={(8,0)}] (0pt, 2pt) -- (0pt,-2pt) node [below] {$\frac{\dim_{\mathrm{H}}\mu}{\overline{\delta}}$};
         
          \draw[shift={(0,3)}] (2pt, 0pt) -- (-2pt,0pt) node [left] {$\dim_{\mathrm{H}}\mu$};
          
          \draw[dotted] (3,0) -- (3,3);
          \draw[dotted] (0,3) -- (3, 3);
          \draw[dotted] (8,0) -- (8,3);
          
        \end{tikzpicture}
\end{center}
\caption{The almost sure dimension $f_{\mu}(\underline{r})$ depicted as a
  "function" of $s_2(\underline{r})$.}
\label{figure:possiblevaluesfordimh}
\end{figure}
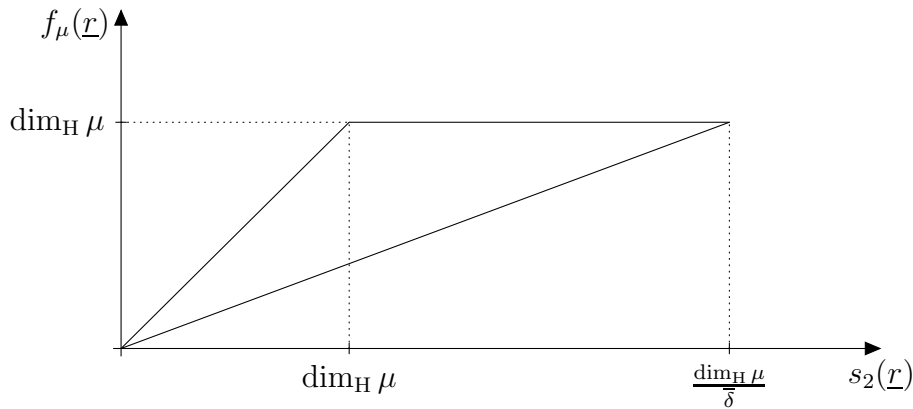
Let $\gamma\in [\frac{s}{u},1]$ and let
$s_0\in\mathopen]0,\frac{s}{\gamma}\mathclose]$.
We will show that there exists a sequence $\underline{r}$ such that
$s_2(\underline{r})=s_0$ and $f_{\mu}(\underline{r})=s_0\gamma$.
Varying $\gamma$ and $s_0$ through their allowed ranges fills the rectangle in
Figure~\ref{figure:possiblevaluesfordimh}.

Define a sequence $(M_j)_{j=1}^{\infty}$ of integers by setting 
\[
M_j:=\left\lfloor \left( \frac{1}{\ell_j^{\gamma}}\right)^{s_0}\right\rfloor
\]
for every $j\in\N\setminus\{0\}$. Set $M_0:=0$.
Consider the sequence $\underline{r}$, where $r_k=\frac{\ell_{j}^{\gamma}}{2}$
for all $k=M_{j-1}+1,\ldots,M_j$. Then $s_2(\underline{r})=s_0$, since 
\[
\sum_{n=1}^\infty r_{n}^{t}=\sum_{j=1}^\infty\sum_{k=M_{j-1}+1}^{M_j}r_k^{t}\approx
\sum_{j=1}^\infty\ell_j^{\gamma(t-s_0)},
\]
and this sum converges when $t>s_0$.
Observe that since $\ell_j^{\gamma}\le\ell_j^{s/u}=L_j$ and the
construction intervals of $C_j$ are at least
$\frac{L_j}{2}$-separated, we have for any $x\in C$ and $M_{j-1}<k\leq M_j$ that
\begin{equation}\label{equation:largerballissmallerball}
C\cap B(x,\tfrac{\ell_j}2)\subset C\cap B(x,r_k)\subset C\cap B(x,\ell_j)
=C\cap B(x,(2r_k)^{\frac 1\gamma}). 
\end{equation}
Thus, by denoting $\rho_k:=(2r_{k})^{\frac 1\gamma}$, we have that 
\begin{equation}\label{equation:upperboundinexample}
  f_{\mu}(\underline{r})\leq s_2(\underline{\rho})=s_2(\underline{r})\gamma
  =s_0\gamma.
\end{equation}
Next we will show that $s_0\gamma\leq f_{\mu}(\underline{r})$. To this end, let
$0<t<s_0\gamma$. Recall that $s_0\gamma\leq s$ and that, for some constant
$D>0$, $\mu(B(x,r))\leq Dr^{s}$ for all $x\in C$ and $r>0$.
In \eqref{equation:largerballissmallerball} we saw that, for all
$x\in C$, the ball $B(x,(2r_k)^{\frac 1\gamma})$ contains exactly one construction
interval $I$ of $C_j$ and $\mu(B(x,r_k))\ge\frac 12\mu(I)$. Therefore,  
\[
I_{t}(\mu_{B(x,r_k)})\lesssim I_t\bigl(\mu_{B(x,(2r_k)^{\frac 1\gamma})}\bigr).
\]
Since
$\mu(B(x,(2r_{k})^{\frac 1\gamma}))\approx l_j^s=
\bigl((2r_{k})^{\frac 1\gamma}\bigr)^{s}$,
Lemma~\ref{lemma:energyupperboundfrostmanmeasures} yields that
\[
I_t\bigl(\mu_{B(x,(2r_k)^{\frac 1\gamma})}\bigr)\lesssim r_{k}^{-\frac t\gamma}.
\]
Note that 
\[
\sum_{k=1}^{\infty}r_k^{\frac t\gamma}=\infty
\]
since $t<s_0\gamma$.
Thus, by Theorem~\ref{theorem:generallowerboundenergy}, we have that 
\[
f_{\mu}(\underline{r})\geq t.
\]
Letting $t\nearrow s_0\gamma$ through a countable sequence yields the desired
lower bound. 
\end{example}

\end{document}